\documentclass[a4paper,oneside,reqno,10pt]{amsart}
\usepackage[latin1]{inputenc}
\usepackage{amsmath, amssymb,amsthm, mathtools, empheq}
  \usepackage{paralist, mathtools}
  \usepackage{graphics} 
  \usepackage{epsfig} 
\usepackage{graphicx}  \usepackage{epstopdf}
 \usepackage[colorlinks=true]{hyperref}
\hypersetup{urlcolor=blue, citecolor=red}

\usepackage{placeins}
\usepackage{enumerate}
\usepackage{comment}

\usepackage{comment}
 
\newtheorem{theorem}{Theorem}[section]

\newtheorem{lemma}[theorem]{Lemma}

\theoremstyle{definition}
\newtheorem{definition}[theorem]{Definition}
\newtheorem{remark}{Remark}

\newcommand{\HH}{\mathit{HH}}
\newcommand{\rr}{\mathbf{R}}

\newcommand{\eff}{{\mathrm{eff}}}
\newcommand{\di}[1]{{\mathrm{div}}\big( #1\big)}

\newcommand{\ve}{\varepsilon}

\newcommand{\ttheta}{\theta_{\ve^2}\big(\frac{x}{\varepsilon}\big)}

\newcommand{\be}{\begin{equation}}
\newcommand{\ee}{\end{equation}}
\newcommand{\ba}{\begin{array}}
\newcommand{\ea}{\end{array}}

\title[Cable equation for myelinated axons]{Derivation of cable equation by multiscale analysis for a model of myelinated axons}


\makeatletter
\g@addto@macro{\endabstract}{\@setabstract}
\newcommand{\authorfootnotes}{\renewcommand\thefootnote{\@fnsymbol\c@footnote}}%
\makeatother

\begin{document}
\maketitle

\begin{center}

\normalsize
\authorfootnotes
Carlos Jerez-Hanckes\textsuperscript{1}, Irina Pettersson\textsuperscript{2}, and Volodymyr Rybalko \textsuperscript{3}
\par \bigskip

\textsuperscript{1} Pontificia Universidad Cat{\'o}lica de Chile, Chile \par
\textsuperscript{2}University of G{\"a}vle, Sweden \par
\textsuperscript{3} Institute for Low Temperature Physics and Engineering, Ukraine \par
\bigskip

\today
\end{center}


\begin{abstract}
The paper concerns the multiscale modeling of a myelinated axon. Taking into account the microstructure with alternating myelinated parts and nodes Ranvier,  we derive a nonlinear cable equation describing the potential propagation along the axon. We assume that the myelin is not a perfect insulator, and assign a low (asymptotically vanishing) conductivity in the myelin. Compared with the case when myelin is assumed to have zero conductivity, an additional potential arises in the limit equation. The coefficient in front of the effective potential contains information about the geometry of the myelinated parts.    
\end{abstract}

\tableofcontents

\section{Introduction}

A nerve impulse is the movement of action potential along a nerve fiber in response to a stimulus, such as touch, pain, heat or cold. It is the way a nerve cell communicates with another cell and makes it act. For example, a signal from the nerve cell might make a muscle cell to contract.  Any disorder in the nervous system can result in a range of symptoms, which include chronic pain, poor coordination, and loss of sensation. Electrical stimulation helps to create neuron activity and to overcome the lost functions of the patients. For example, it is documented that electrical stimulation leads to augmentation of myelin development \cite{li2017electrical} and helps, for example, people with multiple sclerosis and foot drop walk more normally \cite{es-ms}.

The process of excitability of nerve fibers and a mathematical model for the electric current across the axon membrane was presented in the famous work of Hodgkin and Huxley \cite{HH-52}. For their pioneering work in neurophysiology in 1963 the Nobel Prize in Physiology or Medicine was awarded jointly to Sir John Carew Eccles, Alan Lloyd Hodgkin and Andrew Fielding Huxley. 
A typical nerve contains, however, several grouped fascicles, each of them containing many axons. The jump of the potential across the membrane of each individual axon can be modelled in the framework of the Hodgkin-Huxley model, but the alternating myelinated and unmyelinated parts of the membrane present an obvious problem for those attempting to describe its macroscopic response to the electrical stimulation.  In order to model and simulate the respons of biological tissues to electrical stimulation one needs to know how signals propagate along single neurons and, as the next step, how they influence each other in a bundle of axons. 

The signal propagation along a neuron is modelled by a cable equation, usually derived by modeling dendrites and axons as cylinders composed of segments with capacitances and resistances combined in parallel (\cite{HH-52},  \cite{Rall-69}, \cite{Ra-90},  \cite{Ba}, \cite{meffin2014modelling}). The coefficients in such equation depend on the membrane resistances and capacitance of Ranvier nodes and internodes (myelinated parts), as well as on the length of nodes and internodes. There are several works where formal two-scale expansion is applied to a one-dimensional model in order to show that a myelinated neuron can be approximated by a homogeneous cable (\cite{Ba}, \cite{MeLa-08}), but these results do not take into account the microstructure of the fibers, and the geometry of the myelin sheath in particular, as well as they do not justify the formal approximation.

There are many results where the homogenization is applied to cardiac tissue: \cite{NKr}, \cite{PiSa}, \cite{FrPaSc}, \cite{Am}
Cardiac muscle is however fundamentally different  from nerve tissue because the heart is a syncytium. The intracellular space of each cardiac cell is coupled to its neighbor's through intercellular channels. Thus, current can  flow  from the interior of one cell to the interior of another without crossing a cell membrane. 

The present work presents a rigorous derivation of a nonlinear cable equation for signal propagation along a myelinated neuron. We assume that the conductivity of the myelin sheath is small, but not zero, that leads to the appearance of a potential in the limit equation. The potential depends on the geometry of the myelin sheath.


The paper is organized as follows. In Section \ref{sec:setup} we formulate the problem and present the main result in Theorem \ref{th:main}. The rest of the paper is devoted to the proof of Theorem \ref{th:main}. In Sections \ref{sec:apriori} we derive a priori estimates for the potential $u_\ve$ and its jump across the Ranvier nodes. In Section \ref{sec:aux} we construct an auxiliary test function which is used when passing to the limit in Section \ref{sec:pass-to-limit}.

\section{Problem setup}
\label{sec:setup}

\noindent
Let us consider a myelinated axon sparsely suspended in an extracellular medium. 
We assume that the axon has a periodic structure, containing myelinated and unmyelinated parts (nodes of Ranvier) as illustrated on Figure \ref{fig:1}.

\begin{figure}[h]
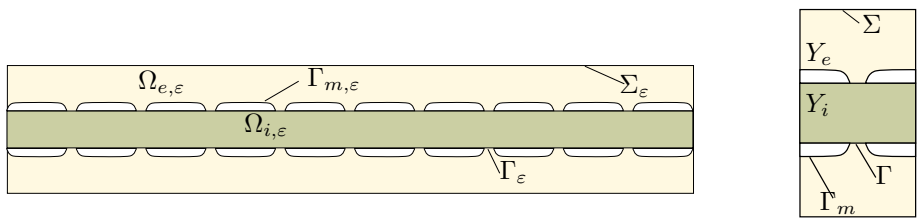

\begin{minipage}[b]{.48\textwidth}
\centering{
\def\svgwidth{1.55\textwidth}
\input{drawing3.pdf_tex}
}
\end{minipage}
\hskip 0.3cm
\begin{minipage}[b]{.48\textwidth}
\flushright{
\def\svgwidth{.35\textwidth}
\input{drawing2.pdf_tex}
}
\end{minipage}
\caption{Simplified geometry of the cross-section of a myelinated axon and the periodicity cell $Y$.}
\label{fig:1}
\end{figure}

A periodicity cell will be denoted by $Y=(-\frac{1}{2}, \frac{1}{2})\times D_{R_0}$ where $D_{R_0}$ is the disk in $\rr^2$ with the radius $R_0$  (see Figure \ref{fig:1}). $Y$ consists of an intracellular part $Y_i=(0,1)\times D_{r_0}$, an extracellular medium $Y_e$, and the myelin sheath $Y_m$ as shown in Fugure \ref{fig:1} (a detailed description of the domain is given in Section \ref{sec:aux}). We denote by $\Gamma_{mi}$ ($\Gamma_{me}$) the interface between $Y_m$ and $Y_i$ ($Y_e$). $\Gamma_m=\Gamma_{mi}\cup \Gamma_{me}$ is the myelinated part of the interface, and $\Gamma$ is the unmyelinated one (surface of a Ranvier node).  The lateral boundary of $Y$ is denoted by $\Sigma$ (we will assume periodicity in $y_1$). We assume that the boundary of the myelin part $\Gamma_m$ is Lipschitz continuous. 
The periodicity cell is then scaled by a small parameter $\ve>0$ and translated along the $x_1$-axis to form a thin periodic cylinder (thickness of order $\ve$) suspended in the extracellular medium (thickness of order $\ve$) with alternating myelinated and unmyelinated parts on the lateral boundary. 

In what follows we denote $x=(x_1, x_2, x_3)=(x_1, x') $ points in $\mathbf R^3$.
 Let
$\Omega_{i,\ve} = (0,L) \times (\ve D_{r_0})$ denote the intracellular domain,
$\Omega_{e,\ve}$ denote the extracellular domain, $\Omega_{m,\ve}$ denote the myelin part, $\Gamma_\ve$ be the unmyelinated part of the boundary, and $\Gamma_{m,\ve}$ be the myelinated one. For simplicity $L$ consists of integer number of periods.

The whole domain $\Omega_\ve = (0,L)\times \ve (-\frac{1}{2}, \frac{1}{2})^2$ is the union of the extracellular, intracellular and myelin domains, and the Ranvier nodes: $\Omega_\ve = \Omega_{i,\ve}\cup \Omega_{e,\ve}\cup \Omega_{m,\ve} \cup \Gamma_\ve$. The lateral part of $\Omega_\ve$ is denoted by $\Sigma_\ve$.

Let $u_\ve^{i}, u_\ve^{e}, u_\ve^m$ denote the electrical potential in the intracellular, extracellular and myelin domains, respectively. 
We assume that the electric potential satisfies homogeneous Neuman boundary conditions on the lateral boundary $\Sigma_\ve$ and homogeneous Dirichlet boundary conditions on the bases $\Gamma_0=\{0\}\times (-\frac{1}{2}, \frac{1}{2})$ and $\Gamma_L=\{L\}\times (-\frac{1}{2}, \frac{1}{2})$.

The transmembrane potential is the jump of the potential through the axon's membrane. We denote it by $[u_\ve] = u_\ve^i - u_\ve^e$. 

Let the conductivity be a piecewise constant function 
\begin{align*}
\sigma_\ve = \left\{
\begin{array}{l}
\sigma_e \quad \mbox{in}\,\, \Omega_{e,\ve},\\
\sigma_i \quad \mbox{in}\,\, \Omega_{i,\ve},\\
\ve^4 \quad \mbox{in}\,\, \Omega_{m,\ve},
\end{array}
\right.
\end{align*}
and $u_\ve$ denote the potential $u_\ve = u_\ve^l$ in $\Omega_{i,\ve}$, $l=i, e, m$.

The potential distribution in $\Omega_\ve$ is described by the following system of equations:

\begin{align}
\label{eq:1-1}
-&\di{\sigma_\ve \Delta u_\ve}  =0, \, &(t,x) &\in (0,T)\times \Omega_\ve\setminus \Gamma_\ve,\\
\label{eq:1-2}
&\sigma_e \nabla u_\ve^e\cdot \nu = -\sigma_i \nabla u_\ve^i\cdot \nu, \, &(t,x) &\in (0,T)\times (\Gamma_\ve\cup \Gamma_{m,\ve}),\\
\label{eq:1-3}
&\ve (c_m \partial_t[u_\ve] + I_{ion}([u_\ve], g_\ve)) = - \sigma_i \nabla u_\ve^i\cdot \nu, \, &(t,x) & \in(0,T)\times  \Gamma_\ve,
\\
\label{eq:1-4}
&\partial_t {g}_\ve  = \HH([u_\ve], g_\ve), \, &(t,x) & \in(0,T)\times  \Gamma_\ve,\\
\label{eq:1-5}
&[u_\ve](x,0) =0, \,\,g_\ve(x,0)=G_0(x_1), \, &x& \in \Gamma_\ve,
\\
\label{eq:1-6}
&\nabla u_\ve^e\cdot \nu   = 0, \, &(t,x)& \in (0,T)\times \Sigma_\ve,
\\
\label{eq:1-7}
&u_\ve  =  0, \, &(t,x)& \in (0,T)\times (\Gamma_0\cup \Gamma_L).
\end{align}
We study the asymptotic behavior of $u_\ve$, as $\ve \to 0$, and derive a one-dimensional effective equation describing the action potential propagation along the axon.

On the Ranvier nodes we assume the continuity of currents (\ref{eq:1-2}), and the Hodgkin-Huxley dynamics for the transmembrane potential (\ref{eq:1-3}). Following the Hodgkin-Huxley model, the applied current through the membrane is a sum of the capacitive current $c_m \partial_t[u_\ve]$, where $c_m$ is the membrane capacitance per unit area, and the ionic current $I_{ion}([u_\ve], g_\ve)$ through the ion channels. 
In the classical Hodgkin-Huxley model there are three types of channel: a sodium channel (Na), a potassium channel (K), and a leakage channel. The conductances of the various ionic fluxes are regulated by the vector of gating variables $g_\ve$.

We assume the homogeneous Dirichlet  boundary condition for $u_\ve^e$ and for $u_\ve^i$ on the bases of the domain, when $x_1=0$ and $x_1=L$; on the lateral boundary of $\Omega_\ve$ we assume the homogeneous Neumann boundary condition; $\nu$ is the unit normal exterior to $\Omega_{e,\ve}$ on $\Sigma$ and $\Gamma_{me}$, and exterior to $\Omega_{i,\ve}$ on $\Gamma_\ve$ and $\Gamma_{mi}$.  Note that $\nu$ on $\Gamma$ is orthogonal to the $x_1$-axes, that is its first component is zero.

We assume that
\begin{itemize}
\item[(H1)]
The function $I_{ion}(v, g)$ is linear w.r.t $v$ and has a form
\begin{align*}
I_{ion}(v, g) = \sum\limits_{j=1}^m H_j(g_{j})(v-v_{r,j}),
\end{align*}
where $g_{,j}$ is the $j$th component of $g$, $v_{r,j}$ is the $j$the component of the resting potential $v_r$, and $H_j$ is positive, bounded, and Lipschitz continuous
\begin{align*}
|H_j(g_1)-H_j(g_2)| \le L_1 |g_1 - g_2|.
\end{align*}
The constant $v_r$ is the reference constant voltage, and $g_\ve$ is a gate variable vector with positive components $0< (g_\ve)_j<1, \,\,j=\overline{1,m}$.

\item[(H2)]
The vector function $\HH(g, v)=F(v) - \alpha g$, where $F$ is Lipschitz continuos
\begin{align*}
|F(v_1)-F(v_2)| \le L_2 |v_1-v_2|.
\end{align*}
\item[(H3)]
$G_0 \in C(0,L)^m$ and takes values between $0$ and $1$ (as the corresponding $g_\ve$). 
\end{itemize}

\begin{remark}
When measuring the respons of a neuron to the external stimulation, one wants to exclude appearance of the action potential in the absence of the external stimulation. To this end one can control the initial state of ionic channels (initial condition for the gate variables) in order to guarantee zero potential at the initial moment. This motivates the choice of zero initial condition for the transmembrane potential $v_\ve$. 
\end{remark}

We will use test function $\phi\in L^\infty(0,T; H^1(\Omega_\ve\setminus \Gamma_\ve))$, $\partial_t \phi \in L^2(0,T; L^2(\Gamma_\ve))$ such that $\phi=0$ for $x_1=0$ and $x_1= L$. The jump of $\phi$ across the Ranvier nodes is denoted by $[\phi]$, $[\phi]=(\phi^i-\phi^e)\Big|_{\Gamma_\ve}$.

The weak formulation corresponding to (\ref{eq:1-1}-\ref{eq:1-7}) is given by:  Find\\
\begin{align*}
u_\ve \in L^\infty(0,T; H^1(\Omega_\ve\setminus \Gamma_\ve)), 
\quad
\partial_t [u_\ve] \in L^2(0,T; L^2(\Gamma_\ve))
\end{align*}
such that $u_\ve =0$ for $x_1 =0$ and $x_1=L$, for any test functions $\phi\in L^\infty(0,T; H^1(\Omega_\ve\setminus \Gamma_\ve))$, $\phi=0$ for $x_1 =0$ and $x_1=L$, and for almost all $t \in (0,T)$
\begin{align}
\label{eq:weak-original-0}
\ve\int_{\Gamma_\ve} c_m \partial_t [u_\ve] [\phi] \, ds
+ \int_{\Omega_\ve\setminus \Gamma_\ve} \sigma_\ve \nabla u_\ve\cdot \nabla \phi \, dx
+ \ve \int_{\Gamma_\ve} I_{ion}([u_\ve], g_\ve)[\phi]\, ds = 0.
\end{align} 
The vector of gate variables $g_{\ve}$ solves the following ordinary differential equation
\begin{align*}
\partial_t g_{\ve} = \HH([u_\ve], g_\ve), \,\, g_\ve(0,x)=G_0(x_1).
\end{align*}
Since $\HH$ is linear with respect to $g_\ve$, we can solve the last ODE and obtain $g_\ve$ as a function (integral functional) of the jump $[u_\ve]$:
\begin{align*}
\langle g_\ve,[u_\ve]\rangle = e^{-\alpha t}\big(G_0(x) + \int_0^t F([u_\ve](\tau, x)) e^{\alpha \tau} \, d\tau\big).
\end{align*}
Substituting this expression into (\ref{eq:weak-original-0}) we obtain the weak formulation of (\ref{eq:1-1})-(\ref{eq:1-7}) in terms of the potential $u_\ve$ and its jump $v_\ve = [u_\ve]$ across $\Gamma_\ve$:
\begin{align}
\label{eq:weak-original}
\ve\int_{\Gamma_\ve} c_m \partial_t v_\ve [\phi] \, ds
+ \int_{\Omega_\ve\setminus \Gamma_\ve} \sigma_\ve \nabla u_\ve\cdot \nabla \phi \, dx
+ \ve \int_{\Gamma_\ve} I_{ion}(v_\ve, \langle g_\ve,v_\ve \rangle)[\phi]\, ds = 0.
\end{align} 

The main result of the paper is given in the following theorem.

\begin{theorem}
\label{th:main}
The transmembrane potential $[u_\ve]$ and the vector of gating variables $g_\ve$ converge uniformly with respect to $t$ in $C(0,T;L^2(\Gamma_\ve))$ to the unique solution $(v_0, g_0)$ of the following one-dimensional problem:
\begin{align}
& c_m \partial_t v_0 + I_{ion}(v_0, g_0) + \overline \Lambda \, v_0  = a^\eff \partial_{x_1x_1}^2 v_0, &(t,x_1) &\in (0,T)\times (0,L),\nonumber\\
\label{eq:hom-prob}
&\partial_t g_0  = \HH(v_0, g_0),  &(t,x_1) &\in (0,T)\times (0,L),\\
&v_0(t,0) = v_0(t,L) =0,&t &\in (0,T),\nonumber\\
&v_0(0,x_1) =0, \,\, g_0(0,x_1)=G_0(x_1),  \hskip -0.5cm & x_1 &\in (0,L). \nonumber
\end{align}
The effective coefficient $a^\eff$ is given by
\begin{align}
\label{eq:eff-coef}
a^\eff = \frac{1}{|\Gamma| |Y|}\left(\big( \sigma^e \int_{Y^e} (\partial_{y_1} N +1)dx\big)^{-1} + \big(\sigma^i |Y^i|\big)^{-1}\right)^{-1},
\end{align}
where the $1$-periodic in $y_1$ function $N$ solves an auxiliary cell problem
\begin{align}
-&\Delta N(y) = 0, &y&\in Y_e, \nonumber\\
&\nabla N \cdot \nu = -\nu_1, \hskip -2.5cm&y&\in \Gamma_m,\label{eq:cell-prob}\\
&\nabla N\cdot \nu = 0,  &y& \in  \Gamma\cup \Sigma,\nonumber\\
&N(y_1,y') \,\,\,\mbox{is periodic in}\,\,\, y_1.\hskip -2.5cm&& \nonumber
\end{align}
The constant $\overline \Lambda$ depends on the geometry of the myelin sheath (see Figure \ref{fig:1-0}) and the conductivities, and is given by
\begin{equation}
\label{eq:Lambda}
\overline\Lambda= \frac{1}{b-a}
\left(\left(\frac{\varphi_A}{\sigma_e(\pi-\varphi_A)}+\frac{\varphi_A}{\sigma_i \pi}\right)^{-1/2}+
\left(\frac{\varphi_B}{\sigma_e(\pi-\varphi_B)}+\frac{\varphi_B}{\sigma_i \pi}\right)^{-1/2}\right).
\end{equation}

\end{theorem}
\bigskip 

\begin{remark}
The effective coefficient $a^\eff$ can be interpreted as the conductivity of the bulk medium corresponding to the conductivity of the intra- and extracellular domains connected in series. 

The effective potential $\overline \Lambda$ is a decreasing function of the angles $\varphi_A, \varphi_B$ and it goes to zero when the angles approach $\pi$. 
\end{remark}
\begin{remark}
Note that, since the equation for $g_0$ is linear in $g_0$, we can solve it explicitly
\begin{align*}
\langle g_0, v\rangle =e^{-\alpha t}\big(G_0(x) + \int_0^t F(v) e^{\alpha \tau} \, d\tau\big).
\end{align*}
Since $F$ is Lipschitz, the composition $I_{ion}(v, g[v])$ is also a Lipschitz function. In this way the effective problem is one nonlinear diffusion equation
\begin{align}
&c_m \partial_t v_0 + I_{ion}(v_0, \langle g_0,v_0\rangle) + \overline \Lambda \, v_0  = a^\eff \partial_{x_1x_1}^2 v_0, \hskip -1cm &(t,x_1) &\in (0,T)\times (0,L),\nonumber\\
\label{eq:1D-eff-without-g}
&v_0(t,0) =  v_0(t,L) =0, &t &\in (0,T),\\
&v_0(0,x_1) =0,  &x_1 &\in (0,L).\nonumber
\end{align}
\end{remark}

\bigskip

To prove Theorem \ref{th:main} we first derive a priori estimates in Section \ref{sec:apriori} (Lemma \ref{lm:apriori-est}), then we prove the two-scale convergence of $u_\ve$ and its gradient (Lemma \ref{lm:convergence-1}) and the 
convergence of $[u_\ve]$ in appropriate spaces (Lemma \ref{lm:piecewise-const-v}). Finally, in Section \ref{sec:pass-to-limit} we pass to the limit in the weak formulation and derive the limit problem (\ref{eq:hom-prob}). Section \ref{sec:aux} is devoted to the construction of an auxiliary function, the main ingredient of the test function used when passing to the limit in the weak formulation.


\section{A priori estimates}
\label{sec:apriori}

\begin{lemma}
There exists a unique
\begin{align*}
u_\ve \in L^\infty(0,T; H^1(\Omega_\ve\setminus \Gamma_\ve)), 
\quad
\partial_t v_\ve = \partial_t [u_\ve] \in L^2(0,T; L^2(\Gamma_\ve))
\end{align*}
such that $u_\ve =0$ for $x_1 =0$ and $x_1=L$, for any test functions $\phi\in L^\infty(0,T; H^1(\Omega_\ve\setminus \Gamma_\ve))$, $\phi=0$ for $x_1 =0$ and $x_1=L$, and for almost all $t \in (0,T)$
\begin{align}
\label{eq:weak-original-0}
\ve\int_{\Gamma_\ve} c_m \partial_t v_\ve [\phi] \, ds
+ \int_{\Omega_\ve\setminus \Gamma_\ve} \sigma_\ve \nabla u_\ve\cdot \nabla \phi \, dx
+ \ve \int_{\Gamma_\ve} I_{ion}(v_\ve, \langle g_\ve, v_\ve \rangle) [\phi]\, ds = 0.
\end{align} 

\end{lemma}
\begin{proof}
The existence of a mild solution follows from the classical semigroup theory (see, for example, \cite{Pazy}). For the existence of more regular solutions see \cite{Matano}, \cite{HJ}. We present just an idea of the proof. 

Denote $v_\ve = [u_\ve]$ and let us rewrite (\ref{eq:1-1})-(\ref{eq:1-7}) in the form
\begin{align}
\label{eq:abstract-eq-v}
\ve (c_m \partial_t v_\ve + I_{ion}(v_\ve, \langle g_\ve, v_\ve\rangle) &= A v_\ve, \quad (t,x)\in(0,T)\times \Gamma_\ve,\\
v_\ve(0,x) &= 0, \quad x \in \Gamma_\ve,
\end{align}
where the operator $A:D(A)\subset L^2(\Gamma_\ve) \to L^2(\Gamma_\ve)$ maps the jump across the nodes $v_\ve=[u_\ve]$ into the solution $u_\ve$  and then to the normal derivative $\sigma_\ve \nabla u_\ve \cdot \nu$. To construct such an operator we fix $f_\ve=- \sigma_\ve \nabla u_\ve \cdot \nu \in L^2(\Gamma_\ve)$, define $v_\ve \in L^2(\Gamma_\ve)$ as a solution of $\ve (c_m \partial_t v_\ve + I_{ion}(v_\ve, \langle g_\ve, v_\ve \rangle)) = f$, and then for each $v_\ve$ we associate a unique solution $u_\ve \in H^1(\Omega_\ve \setminus \Gamma_\ve)$ of problem (\ref{eq:1-1})-(\ref{eq:1-7}). The trace of $u_\ve$ on $\Gamma_\ve$ belongs to $D(A)=H^{3/2}(\Gamma_\ve)$. The operator $A$ is associated with the quadratic form
\begin{align*}
(Av, v)_{L^2(\Gamma_\ve)} = - \int_{\Gamma_\ve} \sigma_\ve \nabla v \cdot \nabla v \, dx,
\end{align*}
is closed and densely defined.  Due to the Poincar\'e inequality, the quadratic form is negative
\begin{align*}
(Av, v)_{L^2(\Gamma_\ve)} = - \int_{\Gamma_\ve} \sigma_\ve \nabla v \cdot \nabla v \, dx \le - C \|v\|_{L^2(\Gamma_\ve)}^2 <0,
\end{align*}
and thus the resolvent set of $A$ contains $\mathbf R_{+}$. Futherrmore, for $\lambda >0$ and $\|v\|_{L^2(\Gamma_\ve)}=1$ we have
\begin{align*}
\lambda < \lambda( v,v)_{L^2(\Gamma_\ve)}  - (Av, v)_{L^2(\Gamma_\ve)} = (\lambda v - A v, v)_{L^2(\Gamma_\ve)},  
\end{align*}  
that implies that $A$ is the infinitesimal generator of a strongly continuous semigroup of contractions (see Theorem 3.1 in \cite{Pazy}).  Since $I_{ion}(v,\langle g, v\rangle)$ is Lipschitz continuous with respect to $v$, there exists a unique mild solution $v_\ve \in C([0,T]; L^2(\Gamma_\ve))$ of (\ref{eq:abstract-eq-v}). It is left to show that $u_\ve\in L^\infty(0,T; H^1(\Omega_\ve \setminus \Gamma_\ve))$ in the bulk domain and $\partial_t v_\ve \in L^2(0, T; L^2(\Gamma_\ve))$. This is done by deriving a priori estimates as in Lemma \ref{lm:apriori-est}.

\end{proof}
\medskip

\begin{lemma}[A priori estimates]
\label{lm:apriori-est}
Let $(u_\ve, g_\ve)$ be a solution of (\ref{eq:1-1}-\ref{eq:1-7}). Denote $v_\ve = [u_\ve]$. Then the following estimates hold:
\begin{enumerate}[(i)]
\item
$\displaystyle \ve^{-1} \int_{\Gamma_\ve} |v_\ve|^2 \, ds \le C, \quad t \in (0,T)$.
\item
$\displaystyle \ve^{-1} \int_0^t \int_{\Gamma_\ve} |\partial_\tau v_\ve|^2 \, ds \, d\tau \le C, \quad t \in (0,T)$.

\item
$\displaystyle \ve^{-2} \int_{\Omega_{i,\ve} \cup \Omega_{e,\ve}}  (|u_\ve|^2 + |\nabla u_\ve|^2) \, dx \le C, \quad t \in (0,T)$.
\item
$\displaystyle \int_{\Omega_{m,\ve}}  |u_\ve|^2 \, dx + \ve^{2} \int_{\Omega_{m,\ve}}  |\nabla u_\ve|^2 \, dx \le C, \quad t \in (0,T)$.

\end{enumerate}
\end{lemma}
\begin{proof}

\noindent
Let us multiply (\ref{eq:1-1}) by $u_\ve$, integrate by parts over $\Omega_\ve \setminus \Gamma_\ve$ and divide the resulting identity by $\ve^2$ (the scaling factor of the order of measure of the thin domain $\Omega_\ve$):
\begin{align*}
\frac{\ve}{2} \frac{d}{dt} \int_{\Gamma_\ve} c_m v_\ve^2 ds + {\ve} \int_{\Gamma_\ve} 
I_{ion}(v_\ve, \langle g_\ve, v_\ve \rangle)
 v_\ve \, ds +
\int_{\Omega_\ve\setminus \Gamma_\ve} \sigma_\ve |\nabla u_\ve|^2 \, dx = 0.
\end{align*}
Integrating the last equality with respect to $t$ we get
\begin{align}
\label{eq:est-1}
\frac{\ve}{2} \int_{\Gamma_\ve} c_m v_\ve^2 ds + \frac{\ve}{2} \int_0^t \int_{\Gamma_\ve} I_{ion}(v_\ve, \langle g_\ve, v_\ve \rangle) v_\ve \, ds d\tau+
\int_0^t \int_{\Omega_\ve\setminus \Gamma_\ve} \sigma_\ve |\nabla u_\ve|^2 \, dx d\tau
= 0.
\end{align}
Using Lipschitz continuity of $I_{ion}$ and applying Gr\"onwall's inequality we obtain the following estimate for $v_\ve$:
\begin{align*}
\ve^{-1} \int_{\Gamma_\ve} v_\ve^2 \, ds \le
C\ve^{-1} (\|G_0\|_{L^2(\Gamma_\ve)}^2 + 1) \, e^{\gamma_0 t} \le C_1(t), 
\end{align*}
for some constants $C, C_1$ and $\gamma_0>0$ in dependent of $\ve$. Estimate $(i)$ is proved. 

From (\ref{eq:est-1})  and $(i)$ we derive an integral estimate for $\nabla u_\ve$:
\begin{align*}
\int_0^t \int_{\Omega_\ve\setminus \Gamma_\ve} \sigma_\ve |\nabla u_\ve|^2 \, dx d\tau \le C.
\end{align*}

Let us now multiply (\ref{eq:1-1}) by $\partial_t u_\ve$ and integrate by parts over $\Omega_\ve\setminus \Gamma_\ve$:
\begin{align*}
\ve^{-1} \int_{\Gamma_\ve} c_m |\partial_t v_\ve|^2 ds + \ve^{-1} \int_{\Gamma_\ve} I_{ion}(v_\ve, \langle g_\ve, v_\ve \rangle) \partial_t v_\ve \, ds+
\frac{\ve^{-2}}{2} \frac{d}{dt} \int_{\Omega_\ve\setminus \Gamma_\ve} \sigma_\ve |\nabla u_\ve|^2 \, dx = 0.
\end{align*}
Integrating w.r.t. $t$ gives
\begin{align}
\label{eq:est-2}
\ve^{-1} \int_0^t \int_{\Gamma_\ve} c_m |\partial_\tau v_\ve|^2 ds d\tau + \ve^{-1} \int_0^t \int_{\Gamma_\ve} I_{ion}(v_\ve, \langle g_\ve, v_\ve \rangle) \partial_\tau v_\ve \, ds d\tau \nonumber \\
+
\frac{\ve^{-2}}{2}  \int_{\Omega_\ve} \sigma_\ve |\nabla u_\ve|^2 \, dx = \frac{\ve^{-2}}{2}  \int_{\Omega_\ve\setminus \Gamma_\ve} \sigma_\ve |\nabla u_\ve|^2\Big|_{t=0} \, dx.
\end{align}
To find $\nabla w_\ve(x)=\nabla u_\ve \big|_{t=0}$ we solve the following elliptic problem
\begin{align*}
-&\di{\sigma_\ve \nabla w_\ve}  =0, \, &x &\in \Omega_\ve\setminus \Gamma_\ve,
\\
& [w_\ve] = v_\ve\Big|_{t=0} = 0, \, &x &\in \Gamma_\ve,
\\
&[\sigma \nabla w_\ve\cdot \nu] = 0, \, &x &\in \Gamma_\ve\cup \Gamma_{m,\ve},
\\
&\nabla w_\ve^e\cdot \nu   = 0, \, &x& \in \Sigma_\ve,
\\
&w_\ve  =  0, \, &x_1& \in \Gamma_0\cup \Gamma_L.
\end{align*}
It is clear that $\nabla w_\ve =0$.

The Gr\"onwalls inequality applied in (\ref{eq:est-2}) yields $(ii)$.

Estimates (\ref{eq:est-2}) and $(ii)$ imply that 
\begin{align}
\label{eq:est-grad}
\ve^{-2}  \int_{\Omega_\ve\setminus \Gamma_\ve} \sigma_\ve |\nabla u_\ve|^2 \, dx \le C, \quad t \in (0,T).
\end{align}
Since $u_\ve$ satisfies the homogeneous Dirichlet boundary condition for $x_1=0$, Friedrichs's inequality is valid for $u_\ve$ in $\Omega_{i, \ve}$ and $\Omega_{e, \ve}$ which gives us $(ii)$.

In order to obtain an $L^2$-bound for $u_\ve$ in $\Omega_{m, \ve}$ we use the Poincar\'e inequality inequality in each myelin part $\ve Y_{m,k}$ and then sum them up to obtain an estimate in $\cup_{k} \ve Y_{m,k} = \Omega_{m, \ve}$. Namely, for $u\in H^1(\ve Y_{m,k})$, let $\bar u_{\ve, k}^m$ denote the mean value over $k$the interface between the intracellular domain and myelin $\Gamma_{mi,k}$
\begin{align*}
\bar u_{\ve,k}^m = \frac{1}{|\ve \Gamma_{mi}|}\int_{\ve \Gamma_{mi,k}} u \, ds.
\end{align*}
We derive with the help of  the  Poincarï¿½ inequality 
\begin{align}
\label{eq:est-3}
\int_{\ve Y_{m,j}} |u_\ve^m - \bar u_{\ve,k}^m|^2\, dx &\le C \ve^2 \|\nabla u_\ve^m\|_{L^2(\ve Y_{m,j})}^2 \nonumber
\\
\int_{\ve Y_{m,j}} |u_\ve^m|^2\, dx &\le C \ve^2( \int_{\ve Y_{m,j}} |\nabla u_\ve^m|^2\, dx + \int_{\ve Y_{m,j}} (\bar u_{\ve,k}^m)^2 dx ).
\end{align}
Due to the continuity of traces of $u_\ve$, 
\begin{align}
\label{eq:est-4}
\int_{\ve Y_{m,j}} (\bar u_{\ve,k}^m)^2 dx = \int_{\ve Y_{m,j}} (\bar u_{\ve,k}^i)^2 dx \le C \ve\|u_\ve^i\|_{L^2(\ve \Gamma_{mi})}^2,\\
\label{eq:est-5}
\ve\|u_\ve^i\|_{L^2(\ve \Gamma_{mi})}^2 \le C(\|u_\ve^i\|_{L^2(\ve Y_{i,j})}^2 + \ve^2 \|\nabla u_\ve^i\|_{L^2(\ve Y_{i,j})}^2).
\end{align}
Combining (\ref{eq:est-3})-(\ref{eq:est-5}) we obtain
\begin{align*}
\int_{\ve Y_{m,j}} |u_\ve^m|^2\, dx &\le C (\ve^2 \int_{\ve Y_{m,j}} |\nabla u_\ve^m|^2\, dx +
\int_{\ve Y_{i,j}} |\nabla u_\ve^i|^2\, dx + \ve^2 \int_{\ve Y_{i,j}} |\nabla u_\ve^i|^2\, dx).
\end{align*}
Adding up $\ve Y_{m,j}$ and taking into account (\ref{eq:est-grad}) yields the estimate for the $L^2$-norm for $u_\ve^m$
\begin{align*}
\int_{\Omega_{m, \ve}} |u_\ve^m|^2\, dx &\le C ( \ve^2 \int_{\Omega_{m, \ve}} |\nabla u_\ve^m|^2\, dx + 
 \int_{\Omega_{i, \ve}} | u_\ve^i|^2\, dx 
 +  \ve^2 \int_{\Omega_{i, \ve}} |\nabla u_\ve^i|^2\, dx)\le C,
\end{align*}
which completes the proof.

\end{proof}

Let us recall the notion of the two-scale convergence that will be used when passing to the limit. 

\begin{definition}
We say that $u_\ve(t,x)$ converges two-scale to $u_0(t, x_1, y)$ in $L^2(0,T; L^2(\Omega_{l,\ve}))$, $l=i, e$, if
\begin{enumerate}[(i)]
\item
$\displaystyle \ve^{-2} \int_0^T \int_{\Omega_{l,\ve}} |u_\ve|^2 dx \, dt < \infty$.
\item For any $\phi(t,x_1) \in C(0,T; L^2(0,L))$, $\psi(y)\in L^2(Y_l)$ we have
\begin{align*}
\lim\limits_{\ve \to 0}\ve^{-2} \int_0^T \int_{\Omega_{l,\ve}} u_\ve(x) \phi(t,x_1) \psi\big(\frac{x}{\ve}\big)\, dx\, dt =
\frac{1}{|Y|}\int_0^T \int_0^L \int_{Y_l} u_0(t, x_1,y) \phi(t,x_1) \psi(y)\, dy \,dx_1\, dt,
\end{align*}
for some function $u_0\in L^2(0,T; L^2((0,L)\times Y))$. 
\end{enumerate} 
\end{definition}

\begin{definition}
We say that $v_\ve(t,x)$ converges two-scale to $v_0(t, x_1, y)$ in $L^2(0,T; L^2(\Gamma_\ve))$ if
\begin{enumerate}[(i)]
\item
$\displaystyle \ve^{-1} \int_0^T \int_{\Gamma_\ve} v_\ve^2 \, ds \, dt < \infty$.
\item For any $\phi(t,x_1) \in L^\infty(0,T; L^2(0,L))$, $\psi(y)\in L^2(\Gamma)$ we have
\begin{align*}
\lim\limits_{\ve \to 0}\ve^{-1} \int_0^T \int_{\Gamma_\ve} v_\ve(x) \phi(t,x_1) \psi\big(\frac{x}{\ve}\big)\, ds_x\, dt =
\frac{1}{|Y|}\int_0^T \int_0^L \int_{\Gamma} v_0(t, x_1, y) \phi(t,x_1) \psi(y)\, ds_y \,dx_1\, dt
\end{align*}
for some function $v_0 \in L^2(0,T; L^2((0,L)\times \Gamma))$.
\end{enumerate} 
\end{definition}

\begin{lemma}
\label{lm:convergence-1}
Let $u_\ve$ be a solution of (\ref{eq:1-1}-\ref{eq:1-7}). Denote by $\mathbf I_{\Omega_{l,\ve}}$ the characteristic functions of $\Omega_{l,\ve}$, $l=i,e$. Then, up to a subsequence, 
\begin{enumerate}[(i)]
\item
$[u_\ve]$ converges two-scale to $v_0(t,x_1,y)$ in $L^2(0,T; L^2(\Gamma_\ve))$.
\vskip 0.3cm
\item
$\partial_t [u_\ve]$ converges two-scale to $\partial_t v_0(t,x_1,y)$ in $L^2(0,T; L^2(\Gamma_\ve))$.
\vskip 0.3cm

\item
$\mathbf I_{\Omega_{l,\ve}} u_\ve$ converges two-scale to $\displaystyle \frac{|Y_l|}{|Y|} \, u_0^l(t,x_1)$ in $L^2(0,T; L^2(\Omega_{l,\ve}))$.
\vskip 0.3cm
\item
$\mathbf I_{\Omega_{l,\ve}} \nabla u_\ve$ converges two-scale to $\displaystyle \frac{1}{|Y|} \, (\partial_{x_1} u_0^l(t,x_1) \mathbf{e_1} + \nabla_y w^l(t, x_1, y)$ in $(L^2(0,T; L^2(\Omega_{l,\ve})))$. Here $\mathbf{e_1} = (1,0,0)\in \rr^3$, $w^l \in L^2(0,T; L^2(0,L)\times H^1(Y))$.
\vskip 0.3cm

\end{enumerate}
\end{lemma}
\begin{proof}
The proof follows the lines of classical compactness results for two-scale convergence and therefore is omitted. We refer to \cite{AlDa-95} for two-scale convergence on periodic surfaces (on $\Gamma_\ve$), to \cite{Zh-2000} and \cite{Pe-17} for two-scale convergence in thin structures and dimension reduction. 
\end{proof}

\begin{lemma}[Properties of {$\left[u_\ve\right]$}]
\label{lm:piecewise-const-v}
Let $u_\ve$ be a solution of (\ref{eq:1-1}-\ref{eq:1-7}). Then there exists a function 
\begin{align*}
\tilde v_\ve(t,x_1) \in L^\infty(0,T; H^1(0,L)) \cap H^1(0,T; L^2(0,L))
\end{align*}
such that
\begin{enumerate}[(i)]
\item
For $t \in (0,T)$, the function $\tilde v_\ve$ approximates $[u_\ve]$:
\begin{align*}
\int_{\Gamma_\ve} |\tilde v_\ve - [u_\ve]|^2 ds \le C \ve \int_{\Omega_{i, \ve} \cup \Omega_{e, \ve}} |\nabla u_\ve|^2 dx.
\end{align*}
\item
There exists $v_0(t,x_1)\in L^\infty(0,T; L^2(0,L))$ such that along a subsequence $\tilde v_\ve$ converges to $v_0(t,x_1)$ uniformly on $[0,T]$, as $\ve \to 0$.
\end{enumerate}
\end{lemma}

\begin{proof}[Proof of Lemma \ref{lm:piecewise-const-v}]

Let us cover $\Omega_\ve$ into a union of overlapping cells $\ve Y_k$ as depicted in Figure \ref{fig:3}. We recall that $\Gamma$ 

\begin{figure}[h]
\centering{
\def\svgwidth{0.7\textwidth}
\input{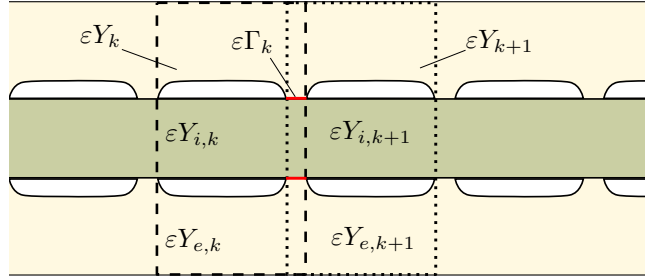}
}
\caption{Overlapping cells $Y_k$ covering $\Omega_\ve$.}\label{fig:3}
\end{figure}

We start by estimating the difference between the mean values of $[u_\ve]$ over $\ve \Gamma_{k}$ and $\ve \Gamma_{k+1}$. Let 
\begin{align*}
\bar u_{\ve, k}^l = \frac{1}{|\ve \Gamma|} \int_{\ve \Gamma_k} u_\ve^l ds, \quad l=i, e.
\end{align*}
For each $\ve Y_{l,k}$, $l=i,e$, we have 
\begin{align*}
\int_{\ve Y_{l, k}} |u_\ve^l - \bar u_{\ve,k}^l|^2 dx \le C\ve^2 \int_{\ve Y_{l, k}} |\nabla u_\ve^l|^2 dx
\end{align*}
owing to the Poicarï¿½ inequality, with $C$ independent of $\ve$. 
Considering traces on $\Gamma_k$ by simple scaling argument one has 
\begin{align}
\label{eq:est-6}
\int_{\ve \Gamma_k}  |u_\ve^l - \bar u_{\ve,k}^l|^2 ds
&\le
C\ve^{-1}\big( \int_{\ve Y_{l, k}}  |u_\ve^l - \bar u_{\ve,k}^l|^2 dx  + \ve^2 \int_{\ve Y_{l, k}} |\nabla u_\ve^l|^2 dx \big) \nonumber\\ 
&\le
C \ve \int_{\ve Y_{l, k}} |\nabla u_\ve^l|^2 dx, \quad l=i, e.
\end{align}
Then the difference between two averages $\bar u_{\ve,k}$ and $\bar u_{\ve,k+1}$ is estimated as follows
\begin{align*}
|\bar u_{\ve,k}^l - \bar u_{\ve,k+1}^l|^2
& \le \frac{2}{|\ve Y_{l,k} \cap \ve Y_{l,k+1}|} \int_{\ve Y_{l,k} \cap \ve Y_{l,k+1}} (|u_\ve^l - \bar u_{\ve,k}^l|^2+ |u_\ve^l - \bar u_{\ve,k+1}^l|^2) dx \\
& \le \frac{C}{\ve}\int_{\ve Y_{l,k} \cup \ve Y_{l,k+1}} |\nabla u_\ve^l|^2 dx. 
\end{align*}
Adding up in $k$ the above estimates we obtain an estimate in $\Omega_{l,\ve}$:
\begin{align}
\label{eq:est-diff-constants}
\sum_k |\bar u_{\ve,k}^l - \bar u_{\ve,k+1}^l|^2 \le \frac{C}{\ve} \int_{\Omega_{l,\ve}} |\nabla u_\ve^l|^2 dx.
\end{align}
Introduce the following notation
\begin{align*}
\bar v_{\ve,k} = \bar u_{\ve,k}^i - \bar u_{\ve,k}^e = \frac{1}{|\ve \Gamma|} \int_{\ve \Gamma_k} [u_\ve] ds.
\end{align*}
Then (\ref{eq:est-6}) and  (\ref{eq:est-diff-constants}) yield
\begin{align}
\label{eq:est-constants}
\int_{\ve \Gamma_k} |[u_\ve] - \bar v_{\ve,k}|^2 ds \le C\ve \int_{\ve Y_{i,k}\cup \ve Y_{e,k}}|\nabla u_\ve|^2 dx,\nonumber\\
\sum_k |\bar v_{\ve,k}- \bar v_{\ve,k+1}|^2 \le \frac{C}{\ve} \int_{\Omega_{i, \ve} \cup \Omega_{e, \ve}} |\nabla u_\ve|^2 dx.
\end{align}
Bounds (\ref{eq:est-constants}) show that $[u_\ve]$ in each cell $\ve Y_k$ is close to a constant $\bar v_{\ve,k}$, and the difference between $\bar v_{\ve,k}$ and $\bar v_{\ve,k+1}$ is small. 

Now we
construct a piecewise linear function $\tilde v_\ve(t, x_1)$ interpolating values $\bar v_{\ve,k}$ linearly and show that
\begin{align}
\label{eq:est-7}
&\int_0^L |\tilde v_\ve|^2 dx_1 \le C, \quad t\in(0,T),\\
\label{eq:est-8}
&\int_0^L | \partial_{x_1} \tilde v_\ve|^2 dx_1 \le C,\quad t\in(0,T),\\
\label{eq:est-9}
&\int_0^T \int_0^L |\partial_t \tilde v_\ve|^2 \, dx_1 dt \le C.
\end{align} 
Indeed, (\ref{eq:est-7}), (\ref{eq:est-8}) follow directly from (\ref{eq:est-constants}):
\begin{align}
\label{eq:est-10}
\int_0^L |\tilde v_\ve|^2 dx_1 & = \sum_k \int_{-\ve/2}^{\ve/2} \big|\frac{\bar v_{\ve,k}+ \bar v_{\ve,k+1}}{2} + x_1 \frac{\bar v_{\ve,k}- \bar v_{\ve,k+1}}{2\ve}\big|^2 dx_1 \nonumber\\
& \le
C\sum_k \ve (|\bar v_{\ve,k}|^2 + |\bar v_{\ve,k+1}|^2) \le C \ve \frac{1}{|\ve \Gamma|} \int_{\ve \Gamma_k} [u_\ve]^2 ds  \notag \\& \le C.
\end{align}
Estimate (\ref{eq:est-8}) is proved in a similar way using (\ref{eq:est-constants}):
\begin{align*}
\int_0^L |\partial_{x_1} \tilde v_\ve|^2 dx_1 & \le C \sum_k \int_{-\ve/2}^{\ve/2} \big|\frac{\bar v_{\ve,k}- \bar v_{\ve,k+1}}{\ve}\big|^2 dx_1 \\& \le \frac{C}{\ve} \sum_k |\bar v_{\ve,k}- \bar v_{\ve,k+1}|^2 \\
& \le C\ve^{-2} \int_{\Omega_{i, \ve} \cup \Omega_{e, \ve}} |\nabla u_\ve|^2 dx \le C.
\end{align*}

Let us prove (\ref{eq:est-9}). Differentiating $\bar v_{\ve,k}$ with respect to $t$, using the the Cauchy-Schwarz inequality yields
\begin{align*}
|\partial_t \bar v_{\ve,k}|^2 = \Big|\frac{1}{|\ve \Gamma_k|} \int_{\ve \Gamma_k} \partial_t [u_\ve]\,ds \Big|^2 
\le \frac{1}{|\ve \Gamma_k|} \int_{\ve \Gamma_k} (\partial_t [u_\ve])^2\,ds.
\end{align*} 
Similarly to (\ref{eq:est-10}), estimate (\ref{eq:est-9}) follows from the last bound and $(ii)$ in Lemma \ref{lm:apriori-est}.

Estimate $(i)$ in the current lemma follows from (\ref{eq:est-constants}). 

The uniform convergence on $(0,T)$ of the constructed piecewise linear approximation is given by the Arzel\`a-Ascoli theorem. 
\begin{theorem}[Arzel\`a-Ascoli theorem]
Let $(X,d)$ be a compact metric space. Then a set $\mathcal F \subset C_0(X; E)$ is precompact (any sequence has a converging subsequence converging uniformly in $X$ to $f\in C_0(0,T; E)$, not necessarily in $\mathcal F$) provided 
\begin{enumerate}
\item
$\mathcal F(x)$ in precompact in $E$, for each $x \in X$.
\item
$\mathcal F$ is equicontinuous at each $x\in X$, that is for all $\gamma>0$ there exists $\delta=\delta(\gamma, x_0)$ so that
\begin{align*}
(\forall x\in X) [d(x,x_0)<\delta \quad \Rightarrow \quad (\forall f \in \mathcal F) \|f(x) - f(x_0)\| <\ve].
\end{align*}
\end{enumerate}
\end{theorem}
\bigskip
The first condition is guaranteed for $\tilde v_\ve$ due to (\ref{eq:est-8}), while the equicontinuity property follows from the bounds (\ref{eq:est-9}):
\begin{align*}
\ve^{-1}\int_{\Gamma_\ve} |\tilde v_\ve(t+\Delta t) - \tilde v_\ve(t) |^2 dx 
= \ve^{-1}\int_{\Gamma_\ve} \int_t^{t+\Delta t} \partial_\tau |\tilde v_\ve(\tau) - \tilde v_\ve(t) |^2 d\tau dx\\
=2\ve^{-1}  \int_{\Gamma_\ve} \int_t^{t+\Delta t} \partial_\tau \tilde v_\ve (\tilde v_\ve(\tau) - \tilde v_\ve(t)) d\tau dx \\
\le C\ve^{-1} \left(\int_{\Gamma_\ve} \int_t^{t+\Delta t} |\tilde v_\ve|^2 ds d\tau\right)^{1/2}
\left(\int_{\Gamma_\ve} \int_t^{t+\Delta t} |\partial_\tau \tilde v_\ve|^2 ds d\tau\right)^{1/2} \\
\le C \ve^{-1} \sqrt{\Delta t} \left(\int_{\Gamma_\ve}  |\tilde v_\ve|^2 ds \right)^{1/2}  \left(\int_{\Gamma_\ve} \int_t^{t+\Delta t} |\partial_\tau \tilde v_\ve|^2 ds d\tau\right)^{1/2} \le C \sqrt{\Delta t}.
\end{align*}
Applying Arzel\`a-Ascoli theorem completes the proof.
\end{proof}
\noindent\makebox[\linewidth]{\rule{\paperwidth}{0.4pt}}

\section{Auxiliary minimization problem}
\label{sec:aux}

We assume that the domains $Y_i$, $Y_m$, $Y_e$ are given in cylindric coordinates $(x_1, r,\phi)$ by
$(x_1,r)\in Y_i^\prime$, $(x_1,r)\in Y_m^\prime$, $(x_1,r)\in Y_e^\prime$. $Y_m^\prime$ is a simply connected domain whose boundary is naturally divided  into two parts $\Gamma_{mi}^\prime=\partial Y_m^\prime\cap \partial Y_i^\prime$ and  $\Gamma_{me}^\prime=\partial Y_m^\prime\cap \partial Y_e^\prime$. The first part is the segment $\{r_0\}\times (a,b)$, while the second one is a smooth curve which never intersects or touches $Y_i^\prime$ except at endpoints 
$A=(a,r_0)$ and $B=(b,r_0)$, and locally near these points it is given by $r=r_a(x_1)$ and $r=r_b(x_1)$.
Moreover, we assume that $r_a$ and $r_b$ are $C^2$-functions whose derivatives do not vanish at 
points $a$ and $b$.

\begin{figure}[h]
\centering{
\def\svgwidth{1.1\textwidth}
\input{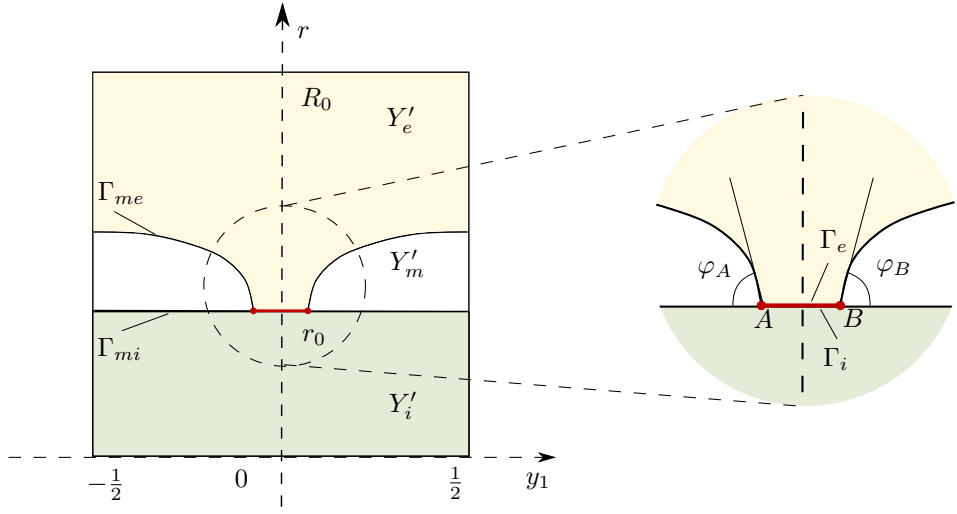}
\label{fig:1-0}
}
\caption{Cross-section of the periodicity cell in the neighborhood of a Ranvier node.}
\end{figure}

Let $\sigma_\delta$ be given by
$$
\sigma_\delta= \begin{cases}
\sigma_i\quad\text{in}\ Y_i,\\
\delta^2\quad\text{in}\ Y_{m},\\
\sigma_{e}\quad\text{in}\ Y_{e}.
\end{cases} 
$$
Consider the minimization problem 
\begin{equation}
\label{MinPrlTheta} 
\lambda_\delta =\inf_{\theta\in H^1_{\rm per}(Y\setminus \Gamma)}  \frac{\displaystyle\int_Y \sigma_\delta |\nabla \theta|^2 dx}{\displaystyle\int_{\Gamma }[\theta]^2ds },
\end{equation}
where the infimum is taken over $1$-periodic in $x_1$-variable functions 
$\theta \in H^1(Y\setminus \Gamma)$, $[\theta]$ denotes the jump of $\theta$ across $\Gamma$,
$[\theta]=\theta_{i}-\theta_{e}$, $\theta_{i}$  and $\theta_{e}$ being limit values (traces) of $\theta$ on $\Gamma$ from $Y_{i}$ and $Y_{e}$, correspondingly. It is easy to see that the infimum in \eqref{MinPrlTheta} is attained on a function $\theta_\delta$ which is defined up to a multiplicative and an additive constant, and $\theta_\delta$ satisfies 
\begin{align}
&{\rm div}\left(\sigma_\delta \nabla \theta_\delta\right)  =0 \quad  \text{in}\ Y\setminus \Gamma, \\
&\quad \left(\sigma_\delta \frac{\partial \theta_\delta}{\partial \nu}\right)_{i} = \left(\sigma_\delta \frac{\partial \theta_\delta}{\partial \nu}\right)_{e} =\lambda_\delta [\theta_\delta]\quad \text{on}\ \Gamma,  \\
&\quad \frac{\partial \theta_\delta}{\partial \nu}=0\quad \text{when}\ |x|=R_0.
\label{rivnyannya} 
\end{align}
Moreover, thanks to the radial symmetry $\theta_\delta =\theta_\delta(x_1,r)$ and
\begin{equation}
\label{MinPrlThetaRadial} 
\lambda_\delta =\frac{\displaystyle\int_{Y^\prime} \sigma_\delta |\nabla_{x_1,r} \theta_\delta|^2\,  r drdx_1}{\displaystyle\int_{\{r_0\}\times (a,b) }[\theta_\delta]^2 r_0 d x_1 }.
\end{equation}

\begin{lemma}
\label{Pervaya_zamechatelnaya_lemma}	
$ $

\begin{enumerate}[(i)]
	 \item
	 The infimum in \eqref{MinPrlTheta} admits the bound
	\begin{equation}
	\label{och_hor_bound}
	\lambda_\delta\leq \Lambda \delta 
	\end{equation}
with $\Lambda>0$ independent of $\delta$. 
\item
 Let $\theta_\delta$ be normalized by 
	\begin{equation}
	\label{th_normaliz}
	\int_{\Gamma}[\theta_\delta]^2 ds=|\Gamma|, \  \int_{Y_e}\theta_\delta dx=0 \quad\text{and}\ \int_{Y_i}\theta_\delta dx\geq 0, 
	\end{equation}	
	then 
	\begin{equation}
	\label{abysho}
	\theta_\delta \rightharpoonup 1\quad\text{weakly in}\ H^1(Y_i),\quad  \theta_\delta \rightharpoonup 0 \quad\text{weakly in}\ H^1(Y_e), \quad \text{as} \ \delta\to 0,
	\end{equation}
	and the following uniform in $\delta>0$ bound holds:
	\begin{equation}
	\label{nuyakrutoi}
	|\theta_\delta|_{L^\infty(Y)}\leq C.
	\end{equation}
\end{enumerate}
\end{lemma}
\begin{proof}
	(i) 
	We begin by constructing an approximation of $\theta_\delta$ away from points $A$ and $B$. There exists a  function $\Theta\in C^2_{\rm loc}(Y^\prime\setminus \Gamma^\prime)$ such that 
$$
0\leq \Theta \leq 1,\text{and}\quad  \Theta=1\quad\text{in}\ Y^\prime_i,\ \Theta=0 \quad\text{in}\ Y^\prime_e,
$$
$$
 \ |\nabla \Theta(x^\prime)|\leq \frac{C}{{\rm dist} (x^\prime,\{A\}\cup\{B\})},    
 \quad \|\nabla^2 \Theta(x^\prime)\|\leq \frac{C}{{\rm dist}^2 (x^\prime,\{A\}\cup\{B\})},  
$$
	where $\|\nabla^2 \Theta\|$ denotes norm of the Hessian of $\Theta$. Since $|\nabla \Theta|$ blows 
	up at points $A$ and $B$ with the rate $1/{\rm dist} (x^\prime,\{A\})$ and 
	$1/{\rm dist} (x^\prime,\{B\})$, any such a function $\Theta$ does not belong to 
	$H^1(Y\setminus \Gamma)$, hence it is to be  corrected near endpoints $A$ and $B$ of $\Gamma^\prime$. For simplicity we assume that in a neighborhood of points $A$ and $B$ the boundary of domain $Y_m^\prime$ is formed by two rays with angles $\varphi_A$ and $\varphi_B$. 
	
	Consider the $\delta$-neighborhood $D_{\delta}(B)$ of the point $B$ and pass to polar coordinates $(\rho,\varphi)$ with the center at $B$. 
	Note that for sufficiently small $\delta$ the set $Y_m^\prime\cap D_{\delta}(B)$ is a circular sector given by $0<\varphi<\varphi_B$ and $0<\rho<\delta$. We set 
\begin{equation}
	\label{CONSRUSTIONOFTESTFUNCTION} 
	\theta_\delta^A =\frac{1}{1-V_\delta}
	\begin{cases}\displaystyle
	\rho^{\alpha_\delta\delta}\cos(\alpha_\delta\delta (\varphi+\pi))/\cos(\alpha_\delta\delta \pi)-V_\delta,\quad \  -\pi<\varphi\leq 0\\
	\displaystyle
	\rho^{\alpha_\delta \delta}\left(1-\frac{\alpha_\delta\sigma_i}{\delta } \tan(\alpha_\delta\delta \pi)\varphi \right)-V_\delta,
	\quad  \  0<\varphi\leq \varphi_B\\
	\displaystyle
	V_\delta \left(\rho^{\alpha_\delta\delta}\cos(\alpha_\delta\delta (\varphi-\pi)/\cos(\alpha_\delta\delta(\varphi_B-\pi))-1\right),
	\quad \  \varphi_B<\varphi< \pi,
	\end{cases}
\end{equation}
with 
\begin{equation}
\label{defofVdelta}
V_\delta:=1-\frac{\alpha_\delta\sigma_i}{\delta } \tan(\alpha_\delta \delta \pi)\varphi_B
\end{equation}
and $\alpha_\delta$ solving the transcendental equation
	\begin{equation}
	\label{CONSRN1}
	\alpha_\delta\sigma_i \tan(\alpha_\delta\delta \pi)=
	\sigma_e \alpha_\delta\tan(\alpha_\delta \delta(\varphi_B-\pi))
	\left(1-\frac{\alpha_\delta \sigma_i}{\delta } \tan(\alpha_\delta\delta \pi)\varphi_B \right).
	\end{equation}
There is a unique solution $\alpha_\delta$ of \eqref{CONSRN1} on $(0, 1/(2\delta))$ and it is asymptotically given by
$$
\alpha_\delta=\frac{1}{\sqrt{\varphi_B}}\sqrt{\frac{1}{\sigma_i \pi} +\frac{1}{\sigma_e(\pi-\varphi_B)}}+O(\delta^2), \,\, \delta \to 0.
$$	
Note that 	$\theta_\delta^B$ is continuous on $\mathbf{R}^2\setminus \mathbf{R}_{-}$ and
\begin{align*}
\lim_{\varphi \to \pm\pi}  \frac{\partial \theta_\delta^B}{\partial \varphi}=0,\quad \text{and}\ 
\lim_{\varphi \to -0} r\sigma_\delta \frac{\partial \theta_\delta^B}{\partial \varphi}=\lim_{\varphi \to +0}  r\sigma_\delta\frac{\partial \theta_\delta^B}{\partial \varphi},\\
\ \lim_{\varphi \to \varphi_B-0}  
		r\sigma_\delta\frac{\partial \theta_\delta^B}{\partial \varphi}=\lim_{\varphi \to  \varphi_B+0}  r\sigma_\delta\frac{\partial \theta_\delta^B}{\partial \varphi}.
\end{align*}

Now  consider the $\delta$-neighborhood of the point $A$ and define the function $\theta_\delta^A$ by replacing $\varphi$ with $\pi-\varphi$ and $\varphi_B$ with $\varphi_A$ in 
\eqref{CONSRUSTIONOFTESTFUNCTION}-\eqref{CONSRN1}, and redefining  $\alpha_\delta$ and $V_\delta$ accordingly.
To glue 
$\Theta$, $\theta_\delta^A$ and $\theta_\delta^B$ together introduce a cut-off function $\chi\in C^\infty(\mathbf{R})$ such that $\chi(\rho)=0$ for $\rho\geq 1$ and $\chi(\rho)=1$ for $\rho\leq 1/2$.  
Set 
$$
\tilde \theta_\delta = (1-\chi(|x^\prime - A|/\delta)-\chi(|x^\prime - B|/\delta))\Theta +
\chi(|x^\prime - A|/\delta)\theta_\delta^A+ 
\chi(|x^\prime - B|/\delta)\theta_\delta^B,
$$
and  use $\tilde \theta_\delta$ as a test function in  \eqref{MinPrlThetaRadial}. Direct computations 
 yield the bound  \eqref{och_hor_bound}. Indeed, by properties of $\Theta$
 \begin{align*}
 \int_{Y^\prime} \sigma_\delta |\nabla_{x_1,r} \tilde \theta_\delta|^2\, r dx_1 dr&=O(\delta^2\log (1/\delta))+
 \int_{D_\delta(A)} \sigma_\delta |\nabla_{x_1,r}\tilde  \theta_\delta|^2\,  r dx_1 dr\\
 &+
\int_{D_\delta(B)} \sigma_\delta |\nabla_{x_1,r} \tilde \theta_\delta|^2\,  r dx_1 dr.
 \end{align*} 
The last two integrals are similar and we consider only the second one:
\begin{align}
\notag
&\int_{D_\delta(B)} \sigma_\delta |\nabla_{x_1,r} \tilde \theta_\delta|^2\,  r dx_1 dr=
\frac{\sigma_i\pi(r_0+O(\delta))}{(1-V_\delta)^2}\int_0^\delta \left(\chi^\prime\big(\frac{\rho}{\delta}\big)\rho^{\alpha_\delta\delta} /\delta +\alpha_\delta\delta\chi\big(\frac{\rho}{\delta}\big)\rho^{\alpha_\delta\delta-1}\right)^2\rho d\rho\\
\notag
&+\frac{\delta^2}{\varphi_B}\int_0^\delta \chi^2(\rho/\delta)\rho^{2\alpha_\delta\delta-1}d\rho\\
\notag
&+\frac{\sigma_e V_\delta^2(\pi-\varphi_B)(r_0+O(\delta))}{(1-V_\delta)^2}\int_0^\delta \left(\chi^\prime(\rho/\delta)(\rho^{\alpha_\delta\delta} /\delta +\alpha_\delta\delta\chi(\rho/\delta)\rho^{\alpha_\delta\delta-1}\right)^2\rho d\rho+O(\delta^2)\\
\notag
&=2r_0\frac{\delta^2}{\varphi_B}\int_0^{\delta/2}\rho^{2\alpha_\delta\delta-1} d\rho+O(\delta^2\log^2(1/\delta))\\
\notag
&= \left(\varphi_A/(\sigma_e(\pi-\varphi_A))+\varphi_A/(\sigma_i \pi)\right)^{-1/2}\delta +O(\delta^2\log^2(1/\delta)).
\end{align}
Also,
 \begin{equation*}
 \int_{\Gamma^\prime}  [\tilde \theta_\delta]^2 d x_1= (b-a) +O(\delta),
 \end{equation*}
thus 
\begin{equation}
\label{Ots_sverhuformula}
\lambda_\delta \leq \overline{\Lambda}\delta+O(\delta^2\log^2(1/\delta)),
\end{equation}
where $\overline \Lambda$ is given by 
\begin{equation}
\label{och_hor_formulaopyat}
\overline\Lambda= \frac{1}{b-a}
\left(\left(\frac{\varphi_A}{\sigma_e(\pi-\varphi_A)}+\frac{\varphi_A}{\sigma_i \pi}\right)^{-1/2}+
\left(\frac{\varphi_B}{\sigma_e(\pi-\varphi_B)}+\frac{\varphi_B}{\sigma_i \pi}\right)^{-1/2}\right).
\end{equation}

(ii) Convergences in \eqref{abysho} easily follow from the bound \eqref{och_hor_bound} and the Poincare  inequality. To prove \eqref{nuyakrutoi} multiply the equation in \eqref{rivnyannya} by $\theta_\delta|\theta_\delta|^{p-2}$, $p\geq 2$, and integrate over $Y\setminus\Gamma$  to find, after integrating by parts, 
$$
(p-1)\int_{Y}\sigma_\delta |\nabla \theta_\delta|^2 |u|^{p-2}=\lambda_\delta
\int_{\Gamma}[\theta_\delta]\,[\theta_\delta |\theta_\delta|^{p-2}]ds.
$$
Therefore we have 
$$
(p-1)\int_{Y}\sigma_\delta |\nabla \theta_\delta|^2 |\theta_\delta|^{p-2}\leq 2\lambda_\delta
\int_{\Gamma_i\cup\Gamma_e}|\theta_\delta|^p ds,
$$
where  $\Gamma_i$ and $\Gamma_e$ denote opposite sides of the surface $\Gamma$.
Thus for $p\geq 2$ it holds 
$$
\frac{p-1}{p^2}\left( \int_{Y}\sigma_\delta |\nabla |\theta_\delta|^{p/2}|^2+
\int_{\Gamma_i\cup\Gamma_e}|\theta_\delta|^p \right)\leq
C\int_{\Gamma_i\cup\Gamma_e}|\theta_\delta|^p
$$
with $C$ independent of $\delta$ and $p\geq 2$. This yields $H^1$-bounds for $|\theta_\delta|^{p/2}$
in $Y_i$ and $Y_e$, that in turn lead to bounds for traces of $|\theta_\delta|^{p/2}$ on $\Gamma_i$ and $\Gamma_e$:
$$
\| |\theta_\delta|^{p/2}\|^2_{H^{1/2}(\Gamma_i\cup\Gamma_e)}\leq
C_1p\|\theta_\delta\|^p_{L^p(\Gamma_i\cup\Gamma_e)}.
$$
Since  $H^{1/2}(\Gamma)$ is continuously embedded in 
$L^{2q}(\Gamma)$  for some $q>1$ (optimal  $q=2$) we have
$$
\| |\theta_\delta|^{p/2}\|^2_{L^{2q}(\Gamma_i\cup\Gamma_e)}\leq
C_2p\|\theta_\delta\|^p_{L^p(\Gamma_i\cup\Gamma_e)}
$$
or
\begin{equation}
\label{iterative}
\| \theta_\delta\|_{L^{pq}(\Gamma_i\cup\Gamma_e)}\leq (C_2p)^{1/p}
\|\theta_\delta\|_{L^p(\Gamma_i\cup\Gamma_e)}.
\end{equation}
It  follows from \eqref{och_hor_bound} and \eqref{th_normaliz} that
 \begin{equation*}
 \| \theta_\delta\|_{L^{2}(\Gamma_i\cup\Gamma_e)}\leq C_3.
\end{equation*}
Then iterative use of \eqref{iterative} yields 
$$
\| u\|_{L^{2q^{k+1}}(\Gamma_i\cup\Gamma_e)}\leq
C_3 \, \exp\big({\frac{1}{2}\sum_0^k\log(2C_2q^j)/q^{j}} \big)
$$
for every integer $k\geq 0$. The series $\sum_1^\infty\log(C_2q^j)/q^{j}$ converges, hence
$\|\theta_\delta\|_{L^\infty(\Gamma_i\cup\Gamma_e)}\leq C$. Finally by the maximum principle $\theta_\delta$ satisfies the same $L^\infty$-bound on $Y\setminus \Gamma$.
\end{proof} 

%

Next we show that the bound \eqref{och_hor_bound}  for $\lambda_\delta$ is in fact precise to the leading order.

\begin{lemma} The following asymptotic result holds:
\begin{equation}
\label{Esche_luchshe_ravenstvo}
\lambda_\delta=\overline{\Lambda} \delta + O\left(\sqrt{\delta^3\log^3 (1/\delta)}\right),
\end{equation}
where $\overline \Lambda$ is given by  \eqref{och_hor_formulaopyat}.
\label{Vtoraya_zamechatelnaya_lemma}
\end{lemma} 	
	
\begin{proof} We use the test function $\tilde \theta_\delta$ constructed in the proof of  Lemma \ref{Pervaya_zamechatelnaya_lemma}. Since normal derivatives  of $\tilde \theta_\delta$ vanish on both sides of $\Gamma^\prime$
and fluxes $r\sigma_\delta \frac{\partial \tilde \theta_\delta}{\partial \nu}$ are continuous across $\partial Y_m^\prime$	we have
\begin{equation}
\label{Vychisleniya_po_chastyam}
0=\int_{Y^\prime\setminus \Gamma^\prime} {\rm div} (r\sigma_\delta\nabla \theta_\delta)\tilde \theta_\delta drdx_1=
\lambda_\delta r_0 \int_{\Gamma^\prime} [\theta_\delta]\, [\tilde \theta_\delta] d x_1+\int_{Y^\prime\setminus \Gamma^\prime} {\rm div} (r\sigma_\delta\nabla\tilde \theta_\delta) \theta_\delta drdx_1.
\end{equation}
It follows from the bound \eqref{och_hor_bound} and normalization conditions \eqref{th_normaliz}
that $\|[\theta_\delta]-1\|_{L^2(\Gamma^\prime)}^2\leq C\delta$, direct calculations also show that
$\|[\tilde \theta_\delta]-1\|_{L^2(\Gamma^\prime)}^2\leq C\delta$, thus
\begin{equation}
\label{Nuochenochevidno}
\int_{\Gamma^\prime} [\theta_\delta]\, [\tilde \theta_\delta] d x_1= (b-a) +O(\delta^{1/2})
\end{equation}
Next we perform asymptotic calculations for the second term in the right hand side of \eqref{Vychisleniya_po_chastyam}. Split the domain $Y^\prime$ into $Z_\delta :=Y^\prime\setminus (D_{\delta}(A)\cup D_{\delta}(B))$  and two disks $D_{\delta}(A)$,  $D_{\delta}(B)$. Since  $\tilde \theta_\delta=\Theta$ in $Z_\delta$, using properties of $\Theta$
and the $L^\infty$-bound \eqref{nuyakrutoi} for $\theta_\delta$ we get 
\begin{equation}
\label{Tolstayaoblast}
\int_{Z_\delta \setminus \Gamma^\prime} {\rm div} (r\sigma_\delta\nabla\tilde \theta_\delta) \theta_\delta drdx_1 =O(\delta^2\log(1/\delta)).
\end{equation}
Next we show that 
\begin{align}
&\int_{D_{\delta}(A)\setminus \Gamma^\prime} {\rm div} (r\sigma_\delta\nabla\tilde \theta_\delta) \theta_\delta drdx_1 + 
\int_{D_{\delta}(B)\setminus \Gamma^\prime} {\rm div} (r\sigma_\delta\nabla\tilde \theta_\delta) \theta_\delta drdx_1\\
&=
r_0\delta	\left(\varphi_A/(\sigma_e(\pi-\varphi_A))+\varphi_A/(\sigma_i \pi)\right)^{-1/2}\notag\\
\label{Melkieoblasti}
&+
r_0\delta\left(\varphi_B/(\sigma_e(\pi-\varphi_B))+\varphi_B/(\sigma_i \pi)\right)^{-1/2}
+O\left(\sqrt{\delta^3\log^3 (1/\delta)}\right).
\end{align}
It suffices to consider only the integral over $D_{\delta}(B)\setminus \Gamma^\prime$. We 
pass to polar coordinates $(\rho, \varphi)$ with the center at $B$ and split the domain $D_{\delta}(B)\setminus \Gamma^\prime$ into the five subdomains:
 $$
 S_{i,1}=\{(\rho, \varphi):\, -\pi< \varphi<0, \  \delta/2\leq\rho<\delta \},\quad
 S_{i,2}=\{(\rho, \varphi):\, -\pi< \varphi<0, \  \rho<\delta/2 \},
 $$
$$
S_{e,1}=\{(\rho, \varphi):\, \varphi_B< \varphi<\pi, \  \delta/2\leq\rho<\delta \},\quad
S_{e,2}=\{(\rho, \varphi):\, \varphi_B< \varphi<\pi, \  \rho<\delta/2 \},
$$
and 
$$
S_{m,1}=\{(\rho, \varphi):\, 0< \varphi<\varphi_B, \  \delta/2\leq\rho<\delta \},\quad
S_{m,2}=\{(\rho, \varphi):\, 0< \varphi<\varphi_B, \  \rho<\delta/2 \}.
$$
The following pointwise bounds hold in these domains:
$$
|{\rm div} (r\sigma_\delta\nabla\tilde \theta_\delta)|=
\begin{cases}
O(\delta^{-1}\log(1/\delta)) \quad \text{in}\ S_{i,1}\ \text{and}\  S_{e,1}\\
O(\delta \rho^{\alpha_\delta \delta-1}) \quad \text{in}\ S_{i,2}\ \text{and}\  S_{e,2} \\
O(1)\quad \text{in}\ S_{m,1}\\
O(\delta^3\rho^{\alpha_\delta\delta-2}) \quad \text{in}\ S_{m,2}.
\end{cases}
$$
Thus,
$$
\int_{D_{\delta}(B)\setminus \Gamma^\prime} {\rm div} (r\sigma_\delta\nabla\tilde \theta_\delta) \theta_\delta drdx_1 =\int_{S_{i,1}\cup S_{e,1}} {\rm div} (r\sigma_\delta\nabla\tilde \theta_\delta) \theta_\delta drdx_1 +O(\delta^2).
$$
Observe that $\theta_\delta$ on $S_{i,1}$ and $S_{e,1}$ is sufficiently close to its mean values  over
$Y_i$ and $Y_e$,
\begin{equation*}
\tau_i:=\frac{1}{|Y_i|}\int_{Y_i}\theta_\delta dx=1+O(\delta^{1/2}) \quad\text{and}\ \tau_e=\frac{1}{|Y_i|}\int_{Y_i}\theta_\delta dx=0,
\end{equation*}
 correspondingly. Namely, by Hardy's inequality
\begin{align}
 \int_{S_{i,1}}|\theta_\delta-\tau_i|^2drdx_1\leq C\delta^2\log(1/\delta) \int_{Y_i}|\nabla \theta_\delta|^2dx\leq C_1 \delta^3 \log(1/\delta), \\
 \quad 
 \int_{S_{e,1}}|\theta_\delta|^2drdx_1\leq  C_2\delta^3 \log(1/\delta).
\end{align}
This leads to the following 
$$
\int_{D_{\delta}(B)\setminus \Gamma^\prime} {\rm div} (r\sigma_\delta\nabla\tilde \theta_\delta) \theta_\delta drdx_1 =
\tau_i \int_{S_{i,1}} {\rm div} (r\sigma_\delta\nabla\tilde \theta_\delta)drdx_1 + O\left(\sqrt{\delta^3\log^3 (1/\delta)}\right).
$$
It remains to calculate the integral in the right hand side integrating by parts 
\begin{align}
\notag\int_{S_{i,1}} {\rm div} (r\sigma_\delta\nabla\tilde \theta_\delta)drdx_1  &=
-\frac{\alpha_\delta \delta \sigma_i  }{1-V_\delta}\int_{-\pi}^{0}(\delta/2)^{\alpha_\delta\delta}
r \frac{\cos(\alpha_\delta\delta (\varphi+\pi))}{\cos(\alpha_\delta\delta \pi)}
d\varphi\\
&+
\frac{\alpha_\delta  \delta \sigma_i r_0}{1-V_\delta}\int_{\delta/2}^\delta  \tan(\alpha_\delta \delta \pi) \chi(\rho/\delta) \frac{\rho^{\alpha_\delta\delta}d\rho}{\rho}\\
&
=-\frac{\alpha_\delta \delta \sigma_i }{1-V_\delta}(\delta/2)^{\alpha_\delta\delta} r_0\pi +O(\delta^2)\\
&
=-r_0\delta	\left(\frac{\varphi_B}{\sigma_e(\pi-\varphi_B)}+\frac{\varphi_B}{\sigma_i \pi}\right)^{-1/2}+O(\delta^2\log(1/\delta)).
\label{nuvrodevychislilos}
\end{align}
This completes the proof of the Lemma.
\end{proof}	

Next we show that, $\theta_\delta $ being normalized by \eqref{th_normaliz}, one has	
\begin{equation}
\label{nasha_svetlaya_tsel}
\frac{1}{\sqrt{\delta}}\sigma_\delta \nabla \theta_\delta \rightharpoonup 0 \quad\text{weakly in} \ L^2(Y). 
\end{equation}
To this end we use the test function $\tilde \theta_\delta$ constructed in the proof of Lemma 
\ref{Pervaya_zamechatelnaya_lemma} to write
\begin{equation}
\int_Y\sigma_\delta |\nabla \theta_\delta|^2 dx-
\int_Y\sigma_\delta |\nabla \theta_\delta|^2 dx  \leq C\sqrt {\delta^3\log^3(1/\delta)},
\label{Energetichotsenka}
\end{equation}
where we have used Lemma 
\ref{Vtoraya_zamechatelnaya_lemma} together with the fact that $\theta_\delta$ minimizes \eqref{MinPrlTheta},  and calculations from the proof of Lemma \ref{Pervaya_zamechatelnaya_lemma}. Representing  $\tilde\theta_\delta$ as 
$\tilde\theta_\delta=(\tilde\theta_\delta-\theta_\delta)+\theta_\delta$ and expanding the left hand side of 
\eqref{Energetichotsenka} we get
\begin{align}
\notag
\int_Y\sigma_\delta |\nabla \tilde\theta_\delta-\nabla\theta_\delta|^2 dx
&\leq C\sqrt {\delta^3\log^3(1/\delta)}-2
\int_Y\sigma_\delta \nabla \theta_\delta \cdot  \nabla(\tilde\theta_\delta-\theta_\delta) dx
\\
&= C\sqrt {\delta^3\log^3(1/\delta)}+2\lambda_\delta
\int_\Gamma [\theta_\delta]  \,[\tilde\theta_\delta-\theta_\delta] dx\\
&\leq C_1\sqrt {\delta^3\log^3(1/\delta)},
\label{Energetichotsenka11}
\end{align}
where we have used \eqref{Nuochenochevidno} to derive the last inequality. Straightforward calculations 
show that $\frac{1}{\sqrt{\delta}}\sigma_\delta \nabla \tilde \theta_\delta \rightharpoonup 0$, which in conjunction 
with  \eqref{Energetichotsenka11} yields \eqref{nasha_svetlaya_tsel}. 

We  summarise properties of 
$\theta_\delta$ in 
\begin{lemma} Let the minimizer $\theta_\delta$ of  \eqref{MinPrlTheta} be normalized by 
\eqref{th_normaliz} then
\begin{enumerate}[(i)]
\item
$
\|\theta_\delta\|_{L^\infty(Y)}\leq C.
$
\vskip 0.2cm
\item
$
\theta_\delta \to \begin{cases}
1 \quad\text{strongly in} \ L^2(Y_i)\\
0 \quad\text{strongly in} \ L^2(Y_e).
\end{cases}
$
\vskip 0.2cm
\item
$
\displaystyle
\frac{1}{\sqrt{\delta}}\sigma_\delta \nabla \theta_\delta \rightharpoonup 0 \quad\text{weakly in} \ L^2(Y). 
$
\vskip 0.2cm
\item
$
[\theta_\delta]\to 1\quad \text{strongly in} \ L^2(\Gamma).
$
\end{enumerate}
\label{Samaya_itogovaya_lemma}
\end{lemma} 	

\begin{lemma}
\label{lm:about-theta}
The rescaled function $\ttheta$ has the following properties:
\begin{enumerate}[(i)]
\item
$\ttheta$ converges to $1$ strongly in $L^2(\Omega_{i, \ve})$ and to $0$ strongly in $L^2(\Omega_{e, \ve})$:
\begin{align*}
\ve^{-2}\int_{\Omega_{i, \ve}} |\ttheta - 1|^2 dx \to 0, \quad \ve \to 0,\\
\ve^{-2}\int_{\Omega_{e, \ve}} |\ttheta|^2 dx \to 0, \quad \ve \to 0.
\end{align*}
\item
$\ttheta$ converges strongly in $L^2(\Gamma_\ve)$ to $1$:
\begin{align*}
\ve^{-1}\int_{\Gamma_\ve} |\ttheta - 1|^2 ds \to 0, \quad \ve \to 0.
\end{align*}
\item
$\displaystyle \ve^{-1} \sigma_\ve \nabla_y \ttheta$ converges weakly two-scale in $L^2(\Omega_{i, \ve} \cup \Omega_{e, \ve})$ to $0$. 
\vskip 0.2cm
\item
$\|\nabla_y\ttheta\|_{L^2(\Omega_{m, \ve})} \le C.$
\end{enumerate}
\end{lemma}
\begin{proof}
\begin{enumerate}[(i)]
\item
Let us prove the convergence in $\Omega_{i, \ve}$. Writing $\Omega_{i, \ve}$ as a union $\cup_k (\ve Y_k)$, rescaling and applying Lemma \ref{Samaya_itogovaya_lemma} we have
\begin{align*}
\ve^{-2}\int_{\Omega_{i, \ve}} |\ttheta - 1|^2 dx = \ve^{-2}\sum_k \int_{\ve Y_{i,k}} |\ttheta - 1|^2 dx \\
=  \ve^{-2}\sum_k \ve^3 \int_{Y_{i,k}} |\theta_\delta - 1|^2 dy =o(1), \quad \ve \to 0.
\end{align*}
The convergence in $\Omega_{e, \ve}$ is proved in the same way.

\item
Similar arguments as above yield
\begin{align*}
\ve^{-1}\int_{\Gamma_\ve} |\ttheta - 1|^2 dx
= \ve^{-1}\sum_k \int_{\ve \Gamma_k} |\ttheta - 1|^2 dx\\ 
= \ve^{-1}\sum_k \ve^2 \int_{\Gamma_k} |\theta_\delta - 1|^2 dy
= o(1), \quad \ve \to 0.
\end{align*}
\item
The convergence to zero follows directly from $(iii)$ in Lemma  \ref{Samaya_itogovaya_lemma}.
\item
Combining (\ref{MinPrlTheta}) and (\ref{och_hor_bound}) one can see that 
\begin{align*}
\int_{Y_m} |\nabla\theta_\delta|^2 dy \le C\delta^{-1}.
\end{align*}
Writing $\Omega_{m, \ve}$ as a union $\cup_k (\ve Y_m)$, rescaling and setting $\delta=\ve^2$ we obtain
\begin{align*}
\int_{\Omega_{m, \ve}} |\nabla_y \ttheta |^2 dx = \sum_k \int_{\ve Y_{m,k}} |\nabla_y \ttheta |^2 dx \\
= \sum_k \ve^3 \int_{Y_{m,k}} |\nabla \theta_\delta(y)|^2 dy \le C.
\end{align*}
\end{enumerate}
\end{proof}



\section{Justification of macroscopic model}
\label{sec:pass-to-limit}

Let us denote $v_\ve=[u_\ve]$. Using Lemmata \ref{lm:convergence-1}, \ref{lm:piecewise-const-v} and \ref{lm:about-theta}, we will pass to the limit in the weak formulation of (\ref{eq:1-1})-(\ref{eq:1-7}):
\begin{align}
\label{eq:weak-pass-to-limit}
\ve^{-1} \int_0^T (c_m \partial_t v_\ve+ I_{ion}(v_\ve, \langle g_\ve,v_\ve\rangle)) [\phi]\, dx dt +
\ve^{-2} \int_0^T \int_{\Omega_\ve\setminus \Gamma_\ve} \sigma_\ve \nabla u_\ve \cdot \nabla \phi \, dx dt =0, 
\end{align}
where $\phi(t,x) \in L^\infty(0,T; H^1(\Omega_\ve \setminus \Gamma_\ve))$ such that $\phi=0$ for $x_1=0$ and $x_1=L$.

For $U_i(t,x_1), U_e(t,x_1) \in C(0,T; H^1(0,L))$ and $U_1(t, x_1,y) \in C(0,T; H^1((0,L)\times Y))$ we construct the following test function:
\begin{align*}
\phi_\ve(t,x) = (U_i(t,x_1) \ttheta + (1-\ttheta)(U_e(t,x_1) + \ve U_1\big(t, x_1, \frac{x}{\ve}\big) ),
\end{align*}
where $\ttheta$ is the auxiliary function introduced in Section \ref{sec:aux}.

Note that due to the strong convergence of the jump of $\ttheta$ (see $(iv)$ Lemma \ref{Samaya_itogovaya_lemma}), the jump of $\phi_\ve$ on $\Gamma_\ve$ converges strongly in $L^2(\Gamma_\ve)$ to $U_i(t,x_1)-U_e(t,x_1)$.
Substituting $\phi_\ve$ into (\ref{eq:weak-pass-to-limit}) we get
\begin{align}
\label{eq:limit-1}
\ve^{-1} &\int_0^T \int_{\Gamma_\ve} (c_m \partial_t v_\ve + I_{ion}(v_\ve, \langle g_\ve,v_\ve\rangle)) [\phi_\ve]\, ds dt \\
\label{eq:limit-2}
+ \ve^{-2} &\int_0^T \int_{\Omega_\ve\setminus \Gamma_\ve} \sigma_\ve \nabla u_\ve \cdot (\ttheta \mathbf{e_1} \partial_{x_1} U_i  + \ve^{-1}U_i \nabla_y \ttheta)\, dx dt\\
\label{eq:limit-3}
+ \ve^{-2}& \int_0^T \int_{\Omega_\ve\setminus \Gamma_\ve} \sigma_\ve \nabla u_\ve \cdot (1-\ttheta)(\mathbf{e_1} \partial_{x_1} U_e + \nabla U_1\big(x_1, \frac{x}{\ve}\big))dx dt \\
 \label{eq:limit-4}
- \ve^{-2}& \int_0^T \int_{\Omega_\ve\setminus \Gamma_\ve} \sigma_\ve \nabla u_\ve \cdot  \ve^{-1} \nabla_y \ttheta (U_e + \ve U_1\big(x_1, \frac{x}{\ve}\big) )  \, dx dt \\
&=I_{1\ve} + I_{2\ve} +I_{3\ve} +I_{4\ve} =0, \nonumber
\end{align}
Let us pass to the limit, as $\ve \to 0$, in each integral $I_{k\ve}$, $k=1,2,3,4$ given by (\ref{eq:limit-1})-(\ref{eq:limit-4}).

Since $[\phi_\ve]$ on $\Gamma_\ve$ converges strongly in $L^2(\Gamma_\ve)$ to $U_i(t,x_1)-U_e(t,x_1)$ and $\partial_t v_\ve$ converges two-scale (weakly) in $L^2(0,T;L^2(\Gamma_\ve))$ and uniformly on $(0,T)$ to $v_0(t,x_1)$, we can pass to the limit in (\ref{eq:limit-1}) and obtain
\begin{align*}
I_{1\ve}  &= \ve^{-1} \int_0^T \int_{\Gamma_\ve} (c_m \partial_t v_\ve + I_{ion}(v_\ve, \langle g_\ve,v_\ve\rangle)) [\phi_\ve]\, ds dt \\
& \xrightarrow[\ve \to 0]{} \quad |\Gamma| \int_0^T \int_0^L (c_m \partial_t v_0 + I_{ion}(v_0, \langle g_0,v_0\rangle )) (U_i - U_e)\, dx_1 dt.
\end{align*} 
Integrating by parts the second integral in (\ref{eq:limit-2}) containing $\nabla_y \ttheta$ and using (\ref{rivnyannya}) and Lemma \ref{Vtoraya_zamechatelnaya_lemma}, we have
\begin{align*}
I_{2\ve} &=\ve^{-2}\int_0^T \int_{\Omega_\ve\setminus \Gamma_\ve} \sigma_\ve \partial_{x_1} u_\ve \, \ttheta \partial_{x_1} U_i \, dx dt \\
&\quad - \ve^{-3} \int_0^T \int_{\Omega_\ve\setminus \Gamma_\ve} \sigma_\ve u_\ve \partial_{y_1} \ttheta \partial_{x_1 x_1}^2 U_i \, dx dt \\
&\quad + \ve^{-3} \int_0^T \int_{\Gamma_\ve} \lambda_{\ve^2} [\ttheta] \, v_\ve\,  U_i\, ds dt \\
&\xrightarrow[\ve \to 0]{} \quad
\frac{|Y_i|}{|Y|}\int_0^T \int_0^L \sigma_i \partial_{x_1} u_0^i \, \partial_{x_1} U_i \, dx_1 dt
+ |\Gamma| \int_0^T \int_0^L  \overline \Lambda\, v_0\, U_i\, dx_1 dt. 
\end{align*}
To pass to the two-scale limit in (\ref{eq:limit-3}) we use (iv) in Lemma \ref{lm:convergence-1} and (i) in Lemma \ref{lm:about-theta} and get
\begin{align*}
I_{3\ve} &=\ve^{-2} \int_0^T \int_{\Omega_\ve\setminus \Gamma_\ve} \sigma_\ve \nabla u_\ve \cdot (1-\ttheta)(\mathbf{e_1} \partial_{x_1} U_e + \nabla_y U_1\big(x_1, \frac{x}{\ve}\big) +  \ve \partial_{x_1}U_1\big(x_1, \frac{x}{\ve}\big))dx dt \\
& \xrightarrow[\ve \to 0]{} \quad
\frac{1}{|Y|} \int_0^T \int_0^L \int_{Y_e} \sigma_e (\mathbf{e_1} \partial_{x_1} u_0^e + \nabla_y w^e) \cdot
(\mathbf{e_1} \partial_{x_1} U_e(t,x_1) + \nabla_y U_1(t,x_1,y))\, dy\, dx_1 dt. 
\end{align*}
Integrating by parts (\ref{eq:limit-4}), using (iii) in Lemma \ref{lm:about-theta}, the interface conditions for $\ttheta$ on $\Gamma_\ve$,  and Lemma \ref{Vtoraya_zamechatelnaya_lemma} yields
\begin{align*}
I_{4\ve}&=\ve^{-2}\int_0^T \int_{\Omega_\ve\setminus \Gamma_\ve} \sigma_\ve \nabla u_\ve \cdot  \ve^{-1} \nabla_y \ttheta (U_e + \ve U_1\big(x_1, \frac{x}{\ve}\big) )  \, dx dt \\
\\
& =-\ve^{-3} \int_0^T \int_{\Omega_\ve\setminus \Gamma_\ve} u_\ve \sigma_\ve \nabla_y \ttheta \cdot \nabla (U_e + \ve U_1\big(x_1, \frac{x}{\ve}\big) )  \, dx dt \\
&\quad  + \ve^{-3} \int_0^T \int_{\Gamma_\ve}\lambda_{\ve^2} [\ttheta] \, v_\ve (U_e + \ve U_1\big(x_1, \frac{x}{\ve}\big) ) \, ds dt\\
&\xrightarrow[\ve \to 0]{} \quad
- |\Gamma| \int_0^T \int_0^L  \overline \Lambda\, v_0\, U_e\, dx_1 dt
\end{align*} 
In this way we obtain a weak formulation of the effective  problem:
\begin{align*}
&|\Gamma| \int_0^T \int_0^L (c_m \partial_t v_0 + I_{ion}(v_0, \langle g_0,v_0\rangle )) (U_i - U_e)\, dx_1 dt\\
+& |\Gamma| \int_0^T \int_0^L  \overline \Lambda\, v_0\, (U_i -U_e)\, dx_1 dt\\
+& \frac{|Y_i|}{|Y|}\int_0^T \int_0^L \sigma_i \partial_{x_1} u_0^i \, \partial_{x_1} U_i \, dx_1 dt\\
+& \frac{1}{|Y|} \int_0^T \int_0^L \int_{Y_e} \sigma_e (\mathbf{e_1} \partial_{x_1} u_0^e + \nabla_y w^e) \cdot
(\mathbf{e_1} \partial_{x_1} U_e(t,x_1) + \nabla_y U_1(t,x_1,y))\, dy\, dx_1 dt=0.
\end{align*}
Computing consequently the variation of the left-hand side of the last equality with respect to $U_1, U_i$ and $U_e$ gives the representation $U_1(t,x_1,y)= N(y)\partial_{x_1} U_e(t,x_1)$, the cell problem (\ref{eq:cell-prob}) and the two one-dimensional equations
\begin{align}
\label{eq:hom-1}
&|\Gamma|(c_m \partial_t v_0 + I_{ion}(v_0, \langle g_0,v_0\rangle ) + \overline \Lambda\, v_0) = \frac{|Y_i|}{|Y|}\,  \sigma_i \partial_{x_1 x_1}^2 \, u_0^i,\\
\label{eq:hom-2}
&|\Gamma|(c_m \partial_t v_0 + I_{ion}(v_0, \langle g_0,v_0\rangle ) + \overline \Lambda\, v_0) = - \frac{1}{|Y|} \int_{Y_e} \sigma_e |\mathbf{e_1} + \nabla_y N|^2  \, \partial_{x_1 x_1}^2 u_0^e \,dy.
\end{align}
Introducing (\ref{eq:eff-coef}) and adding up (\ref{eq:hom-2}) and (\ref{eq:hom-1}) yield (\ref{eq:1D-eff-without-g}). The proof of Theorem \ref{th:main} is complete.

\section{Acknowledgements}
This research was supported by the Swedish Foundation for International Cooperation in Research and Higher Education
STINT (research grant IB 2017-7370).

\end{document}